\documentclass{amsart}
\usepackage{amsmath,amssymb,amsthm}
\newtheorem{theorem}{Theorem}
\newtheorem{lemma}[theorem]{Lemma}
\newtheorem{proposition}[theorem]{Proposition}
\newtheorem{corollary}[theorem]{Corollary}

\newtheorem{remark}[theorem]{Remark}
\newtheorem{example}[theorem]{Example}

\begin{document}
\title{FP-injectivity of factors of injective modules}
\author{Fran\c{c}ois Couchot}
\address{Laboratoire de Math\'ematiques Nicolas Oresme, CNRS UMR
  6139,
D\'epartement de math\'ematiques et m\'ecanique,
14032 Caen cedex, France}
\email{couchot@unicaen.fr}  

\begin{abstract} It is shown that a ring is left semihereditary if and only each homomorphic image of its injective hull as left module is FP-injective. It is also proven that a commutative ring $R$ is reduced and arithmetical if and only if $E/U$ if FP-injective for any FP-injective $R$-module $E$ and for any  submodule $U$ of finite Goldie dimension. A characterization of commutative rings for which each module of finite Goldie dimension is of injective dimension at most one is given. Let $R$ be a chain ring and $Z$ its subset of zerodivisors. It is proven that $E/U$ is FP-injective for each FP-injective $R$-module $E$ and each pure polyserial submodule $U$ of $E$ if $R/I$ is complete in its f.c. topology for each ideal $I$ whose the top prime ideal is $Z$. The converse holds if each indecomposable injective module whose the bottom prime ideal is $Z$ contains a pure uniserial submodule. For some chain ring $R$ we show that $E/U$ is FP-injective for any FP-injective module $E$ and any its submodule $U$ of finite Goldie dimension, even if $R$ is not coherent. It follows that any Archimedean chain ring is either coherent or maximal if and only if each factor of any injective module of finite Goldie dimension modulo a pure submodule is injective.
\end{abstract}

\maketitle

\section{Introduction}

It is well known that each factor of a divisible module over an integral domain is divisible. By \cite[Proposition IX.3.4]{FuSa01} an integral domain is {\bf Pr\"ufer} (each ideal is flat) if and only if each  divisible module is FP-injective. So, over any Pr\"ufer domain each factor module of a FP-injective module is FP-injective too. More generally, a ring $R$ is left {\bf hereditary} (each left ideal is projective) if and only if (by \cite[Proposition I.6.2]{CaEi56}) each factor of any injective left $R$-module is injective, a ring $R$ is left {\bf semihereditary} (each finitely generated left ideal is projective) if and only if (by \cite[Theorem 2]{Meg70}) each factor of any FP-injective left $R$-module is FP-injective,
By \cite[Th\'eor\`eme 4]{Cou75} a commutative ring $R$ has {\bf global weak dimension} $\leq 1$ (each ideal is flat) if and only if each finitely cogenerated factor of any finitely cogenerated injective module is FP-injective, and in this case, by using \cite[Th\'eor\`emes 3 et 4]{Cou75} it is possible to show that each factor of any FP-injective module modulo a submodule of finite Goldie dimension is FP-injective. In \cite[Theorem 2.3]{Fac82} there is a characterization of commutative rings for which each factor of any finitely cogenerated injective module is injective. On the other hand, by using \cite[Theorem 3.2]{Ste70} it is not difficult to show that a ring $R$ is left {\bf coherent} (each finitely generated left ideal is finitely presented) if and only if each factor of any FP-injective left $R$-module modulo a pure submodule is FP-injective (each direct limit of a system of FP-injective modules is factor of the direct sum of all FP-injective modules of the system modulo a pure submodule).

\bigskip

In this paper the following two questions are studied:
\begin{itemize}
\item What are the rings $R$ for which $E/U$ is FP-injective for any FP-injective left module $E$ and any submodule $U$ of finite Goldie dimension?
\item What are the rings $R$ for which any left module of finite Goldie dimension is of injective dimension at most one? 
\end{itemize}
A complete answer to these questions is given but only when $R$ is commutative. However, a result in the general case is given by extending Problem 33 posed by Fuchs and Salce in \cite[p. 306]{FuSa01} and solved by Laradji in \cite{Lar05}. 

Then, we examine the following question:
\begin{itemize}
\item  What are the rings $R$ for which $E/U$ is FP-injective for any FP-injective left module $E$ and any pure submodule $U$ of finite Goldie dimension?
\end{itemize}
We study this question uniquely in the case where $R$ is a commutative chain ring, and even in this case, it is not easy to get some interesting results.

\bigskip

In this paper all rings are associative and commutative (except at the beginning of section \ref{S:glo}) with unity and
all modules are unital. First we give some definitions. 

An $R$-module $M$ is said to be \textbf{uniserial} if its set of submodules is totally ordered by inclusion and  $R$ is a \textbf{chain ring}\footnote{we prefer ``chain ring '' to ``valuation ring'' to avoid confusion with ``Manis valuation ring''.} if it is uniserial as $R$-module. In the sequel, if $R$ is a chain ring, we denote by $P$ its maximal ideal, $Z$ its subset of zerodivisors and $Q(=R_Z)$ its quotient ring. Recall that a chain ring $R$ is said to be {\bf Archimedean} if $P$ is the sole non-zero prime ideal.

A module $M$ has \textbf{finite Goldie dimension} if its injective hull is a finite direct sum of indecomposable injective modules. A module $M$ is said to be \textbf{finitely cogenerated} if its injective hull is a finite direct sum of injective hulls of simple modules. The \textbf{f.c. topology} on a module $M$ is the linear topology defined by taking as a basis of neighbourhoods of zero all submodules $G$ for which $M/G$ is finitely cogenerated (see \cite{Vam75}). This topology is always Hausdorff. When $R$ is a chain ring which is not a finitely cogenerated $R$-module, the f.c. topology on $R$ coincides with the $R$-topology which is defined by taking as a basis of neighbourhoods of zero all non-zero principal ideals.
A module $M$ is called \textbf{linearly compact} if any family of cosets having the finite intersection property has nonempty intersection.

A ring $R$ is said to be \textbf{(almost) maximal} if $R/A$ is linearly compact for any (non-zero) proper ideal $A$.

An exact sequence \ $0 \rightarrow F \rightarrow E \rightarrow G \rightarrow 0$ \ is \textbf{pure}
if it remains exact when tensoring it with any $R$-module. In this case
we say that \ $F$ \ is a \textbf{pure} submodule of $E$.

We say that an $R$-module
$E$ is  \textbf{FP-injective} if $\mathrm{Ext}_R^1(F,E)=0,$ for every finitely
presented $R$-module $F.$ A ring $R$ is called \textbf{self
FP-injective} if it is FP-injective as $R$-module.

\section{Global case}
\label{S:glo}
\begin{proposition}\label{P:redu}
Let $R$ be a ring, $E$ a left $R$-module and $U$ a submodule of $E$. Then the following conditions are equivalent:
\begin{enumerate}[(1)]
\item $E/U$ is FP-injective if $E$ is FP-injective;
\item $E/U$ is FP-injective if $E$ is an injective hull of $U$.
\end{enumerate}
\end{proposition}
\begin{proof}
It is obvious that $(1)\Rightarrow (2)$.

$(2)\Rightarrow (1)$. First we assume that $E$ is injective. Then $E$ contains a submodule $E'$ which is an injective hull of $U$. Since $E/E'$ is injective and $E'/U$ FP-injective, then $E/U$ is FP-injective too. Now we assume that $E$ is FP-injective. Let $H$ be the injective hull of $E$. Then $E/U$ is a pure submodule of $H/U$. We conclude that $E/U$ is FP-injective.
\end{proof}
The following theorem contains a generalization of \cite[Corollary 4]{Lar05}.
\begin{theorem}
\label{T:semiher} Let $R$ be a ring and $I$ its injective hull as left $R$-module. Then the following conditions are equivalent:
\begin{enumerate}
\item $R$ is left semihereditary;
\item each homomorphic image of any FP-injective left module is FP-injective;
\item each homomorphic image of $I$ is FP-injective.
\end{enumerate}
\end{theorem}
\begin{proof}
By \cite[Theorem 2]{Meg70} $(1)\Leftrightarrow (2)$, and it is obvious that $(2)\Rightarrow (3)$.

$(3)\Rightarrow (2)$. Let $M$ be a FP-injective left $R$-module and $K$ a submodule of $M$. To show that $M/K$ is FP-injective we may assume that $M$ is injective by Proposition \ref{P:redu}. There exist a set $\Lambda$ and an epimorphism $g:R^{(\Lambda)}\rightarrow M$. Since $M$ is injective, we can extend $g$ to an epimorphism from $I^{(\Lambda)}$ into $M$. Hence, it is enough to show that each homomorphic image of $I^{(\Lambda)}$ is FP-injective for any set $\Lambda$. First we assume that $\Lambda$ is a finite set of cardinal $n$. Let $K$ be a submodule of $I^n$ and $p:I^n=I^{n-1}\oplus I\rightarrow I$ the canonical projection. We note $K'$ the image of $K$ by $p$. We get the following exact sequence:
\[0\rightarrow I^{n-1}/K\cap I^{n-1}\rightarrow I^n/K\rightarrow I/K'\rightarrow 0.\]
So, by induction on $n$ we get that $I^n/K$ is FP-injective.
 Now, let $(\Lambda_{\gamma})_{\gamma\in\Gamma}$  be the family of finite subsets of $\Lambda$ where $\Gamma$ is an index set. For each $\gamma\in\Gamma$ we put 
\[I_{\gamma}=\{x=(x_{\lambda})_{\lambda\in\Lambda}\in I^{(\Lambda)}\mid x_{\lambda}=0,\ \forall \lambda\notin\Lambda_{\gamma}\}.\] If $K$ is submodule of $I^{(\Lambda)}$ then $I^{(\Lambda)}/K$ is the union of the family of submodules $(I_{\gamma}/K\cap I_{\gamma})_{\gamma\in\Gamma}$. We use \cite[Corollary 2.3]{Ste70} to conclude.
\end{proof}

Given a ring $R$ and a left $R$-module $M$, we say that $M$ is \textbf{P-injective} if $\mathrm{Ext}_R^1(R/Rr,M)=0$ for any $r\in R$. When $R$ is a domain, $M$ is P-injective if and only if it is divisible. We say that $R$ is a left \textbf{PP-ring} if  any principal left ideal is projective.
 
The following theorem can be proven in a similar way as the previous.

\begin{theorem}
\label{T:PP} Let $R$ be a ring and $I$ its injective hull as left $R$-module. Then the following conditions are equivalent:
\begin{enumerate}
\item[{\rm (1)}] $R$ is a left PP-ring;
\item[{\rm (2)}] each homomorphic image of any P-injective left module is P-injective;
\item[{\rm (3)}] each homomorphic image of $I$ is P-injective.
\end{enumerate}
\end{theorem}

The following is a slight improvement of \cite[Th\'eor\`eme 4]{Cou75}. 

\begin{theorem}\label{T:gw1}
Let $R$ be a commutative ring. The following conditions are equi\-valent:
\begin{enumerate}
\item $R$ is of global weak dimension $\leq 1$;
\item each finitely cogenerated factor of any finitely cogenerated FP-injective $R$-module is FP-injective;
\item each finitely cogenerated $R$-module is of FP-injective dimension $\leq 1$;
\item each finitely cogenerated factor of any FP-injective $R$-module of finite Goldie dimension is FP-injective;
\item  each $R$-module of finite Goldie dimension is of FP-injective dimension $\leq 1$.
\end{enumerate}
\end{theorem}
\begin{proof}
By \cite[Th\'eor\`eme 4]{Cou75} $(1)\Leftrightarrow (2)$. It is obvious that $(3)\Rightarrow (2)$, $(5)\Rightarrow (4)$ and $(4)\Rightarrow (2)$.

$(2)\Rightarrow (3)$. Let $E$ be a injective $R$-module of finite Goldie dimension and $M$ be a factor of $E$. By using \cite[Th\'eor\`eme 3]{Cou75}, it is easy to prove that $M$ is a pure submodule of an module $M'$ with $M'=\prod_{\lambda\in\Lambda}M_{\lambda}$, where $\Lambda$ is an index set and $M_{\lambda}$ is a finitely cogenerated factor of $M$ for each $\lambda\in\Lambda$. Then $M_{\lambda}$ is a factor of $E$, whence it is FP-injective by $(2)$, for each $\lambda\in\Lambda$. We successively deduce that $M'$ and $M$ are FP-injective.

$(4)\Rightarrow (5)$. By Proposition \ref{P:redu} we may assume that $E$ is injective of finite Goldie dimension. To conclude we do as in the proof of $(2)\Rightarrow (3)$.

$(1)\Rightarrow (4)$. Let $p:E\rightarrow M$ be an epimorphism where $E$ is an injective $R$-module of finie Goldie dimension and $M$ a finitely cogenerated $R$-module. Let $u$ be the inclusion map from $M$ into its injective hull $F$ and $f=u\circ p$. Then $E=E_1\oplus\dots\oplus E_n$ and $F=F_1\oplus\dots\oplus F_q$ where $E_i$ and $F_j$ are indecomposable for $i=1,\dots,n$ and $j=1,\dots,q$. Since the endomorphism ring of any indecomposable injective module is local, there exist maximal ideals $P_1,\dots,P_n$ and $L_1,\dots,L_p$ of $R$ such that $E_i$ is a module over $R_{P_i}$ for $i=1,\dots,n$ and $F_j$ is a module over $R_{L_j}$ for $j=1,\dots,q$. Let $S=R\setminus(P_1\cup\dots\cup P_n\cup L_1\cup\dots\cup L_q)$. Then $E$ and $F$ are modules over $S^{-1}R$, $f$ is a $S^{-1}R$-homomorphism. It follows that $M$ is also a module over $S^{-1}R$. Since $S^{-1}R$ is semilocal, $(1)$ implies that it is semihereditary. We conclude that $M$ is FP-injective.
\end{proof}

Recall that a commutative ring $R$ is said to be {\bf arithmetical} if $R_P$ is a chain ring for each maximal ideal $P$ of $R$. It is well known that a reduced ring is arithmetical if and only if it is of global weak dimension $\leq 1$.

\begin{theorem}\label{T:gw2}
Let $R$ be a commutative ring. The following conditions are equi\-valent:
\begin{enumerate}
\item $R$ is of global weak dimension $\leq 1$ and $R/L$ is  an almost maximal Pr\"ufer domain for every minimal prime ideal $L$ of $R$;
\item $R$ is of global weak dimension $\leq 1$ and each factor of $R_L$ is injective for each minimal prime ideal of $R$;
\item each $R$-module of finite Goldie dimension is of injective dimension $\leq 1$.
\end{enumerate} 
\end{theorem}
\begin{proof}
Assume that $R$ is a reduced arithmetical ring. If $L$ is a minimal prime ideal of $R$, then $R/L$ is a submodule of $R_L$ and consequently it is a flat $R$-module. So, each injective $R/L$-module is injective over $R$ too. By \cite[Proposition IX.4.5]{FuSa01} we conclude that $(1)\Leftrightarrow (2)$.

$(3)\Rightarrow (2)$. By Theorem \ref{T:gw1} $R$ is a reduced arithmetical ring. Let $L$ be a minimal prime ideal. Then $R_L$ is a field and so it is an injective module of Goldie dimension one.

$(2)\Rightarrow (3)$. Let $I$ be an indecomposable injective module, $P$ the prime ideal of $R$ which is the inverse image of the maximal ideal of $\mathrm{End}_R(I)$ by the natural map $R\rightarrow\mathrm{End}_R(I)$ and $L$ the minimal prime ideal of $R$ contained in $P$. Since $I$ is a module over $R_P$ then it is annihilated by $L$, and since $R_P$ is almost maximal it is a factor of $R_L$. Now let $U$ be a module of finite Goldie dimension and $E$ its injective hull. Then $E=I_1\oplus\dots\oplus I_n$ where $I_i$ is indecomposable for $i=1,\dots,n$. Let $L_1,\dots,L_p$ be the minimal prime ideals of $R$ such that, for each $i=1,\dots,n$ there exists $j$, $1\leq j\leq p$ such that $I_i$ is annihilated by $L_k$. Then $E$ is annihilated by $L=L_1\cap\dots\cap L_p$. Since $R$ is arithmetical the minimal prime ideals $L_1,\dots, L_p$ are comaximal. Then $E=E_1\oplus \dots\oplus E_p$, $U=U_1\oplus\dots\oplus U_p$ where $E_k=E/L_kE$, $U_k=U/L_kU$ for $k=1,\dots,p$. So, $E/U\cong E_1/U_1\oplus\dots\oplus E_p/U_p$. From above, for each $k=1,\dots,p$, we deduce that $E_k/U_k$ is a factor of $R_{L_k}^{m_k}$ for some positive integer $m_k$. By induction on $m_k$ we show that $E_k/U_k$ is injective. Hence $E/U$ is injective.
\end{proof}

\begin{example} Let $R$ be the B\'ezout domain due to Heinzer and Ohm constructed in \cite[Example III.5.5]{FuSa01}. Then the injective dimension of any finitely cogenerated $R$-module is at most one, but $R$ does not verify the equivalent conditions of Theorem \ref{T:gw2}.
\end{example}
\begin{proof}
Since $R_P$ is a Noetherian valuation domain, it is almost maximal and each non-zero prime ideal is contained in a unique maximal ideal. So, by \cite[Theorem 2.3]{Fac82} each finitely cogenerated $R$-module is of injective dimension $\leq 1$. But some elements
of $R$ are contained in infinite many maximal ideals. So, by \cite[Theorem IV.3.9]{FuSa01} $R$ is not an almost maximal domain.
\end{proof}

\begin{proposition}\label{P:locoh}
Let $R$ be a locally coherent commutative ring. For any FP-injective $R$-module $E$ and any pure submodule $U$ of finite Goldie dimension, $E/U$ is FP-injective.
\end{proposition}
\begin{proof}
By Proposition \ref{P:redu} we may assume that $E$ is injective of finite Goldie dimension. If $I$ is an indecomposable injective module then $\mathrm{End}_R(I)$ is a local ring. Let $P$ be the prime ideal which is the inverse image of the maximal ideal of $\mathrm{End}_R(I)$ by the canonical map $R\rightarrow\mathrm{End}_R(I)$. It follows that $I$ is a module over $R_P$. Now let $E=\oplus_{k=1}^nI_k$ be a $R$-module where $I_k$ is indecomposable and injective for $k=1,\dots,n$. Let $P_k$ be the prime ideal defined as above by $I_k$ for $k=1\dots,n$ and let $S=R\setminus (\cup_{1\leq k\leq n}P_k)$. Then $E$ and $U$ are module over the semilocal ring $S^{-1}R$. Since $R$ is locally coherent then $S^{-1}R$ is coherent. It follows that $E/U$ is FP-injective.
\end{proof}

\section{Chain ring case: preliminaries}
Some preliminaries are needed to prove our main results: Proposition \ref{P:main} and Theorems \ref{T:main} and \ref{T:mainCoh}.

\begin{lemma}
\label{L:crucial} Let $R$ be a chain ring, $E$ a FP-injective module, $U$ a pure essential submodule of $E$, $x\in E\setminus U$ and $a\in R$ such that $(0:a)\subseteq (U:x)$. Then:
\begin{enumerate}
\item if $(0:a)\subset (U:x)$ then $x\in U+aE$;
\item if $(0:a)=(U:x)$ then $x\notin U+aE$.
\end{enumerate}
\end{lemma}
\begin{proof}
$(1)$. Let $b\in (U:x)\setminus (0:a)$. Then $bx\in U$. Since $U$ is a pure submodule  there exists $u\in U$ such that $bx=bu$. We get that $(0:a)\subset Rb\subseteq (0:x-u)$. The FP-injectivity of $E$ implies that there exists $y\in E$ such that $x-u=ay$.

$(2)$. By way of contradiction suppose there exist $u\in U$ and $y\in E$ such that $x=u+ay$. Then we get that $(U:x)=(U:x-u)=(0:x-u)$. So, $U\cap R(x-u)=0$. This contradicts that $E$ is an essential extension of $U$.
\end{proof}

Let $M$ be a non-zero module over a  ring $R$. As in \cite[p.338]{FuSa01}
 we set:
\[M_{\sharp}=\{s\in R\mid \exists 0\ne x\in M\ \mathrm{such\ that}\
sx=0\}\quad\mathrm{and}\quad M^{\sharp}=\{s\in R\mid sM\subset M\}.\] 
Then $R\setminus M_{\sharp}$ and $R\setminus M^{\sharp}$ are multiplicative subsets of $R$. 

If $M$ is a module over a chain ring $R$ then $M_{\sharp}$ and $M^{\sharp}$ are prime ideals and they are called the {\bf bottom} and the {\bf top prime ideal}, respectively, associated with $M$.

When $I$ is a non-zero proper ideal, it is easy to check that \[I^{\sharp}=\{s\in R\mid I\subset (I:s)\}.\] So, $I^{\sharp}$ is the inverse image of the set of zero-divisors of $R/I$ by the canonical epimorphism $R\rightarrow R/I$. If we extend this definition to the ideal $0$ we have $0^{\sharp}=Z$. A proper ideal  $I$ of a chain ring $R$ is said to be \textbf{Archimedean} if $I^{\sharp}=P$. When $R$ is Archimedean each non-zero ideal of $R$ is Archimedean.

\begin{remark}\label{R:P=Z}
\textnormal{If $P=Z$ then by \cite[Lemma 3]{Gil71} and \cite[Proposition 1.3]{KlLe69} we have $(0:(0:I))=I$ for each ideal $I$ which is not of the form $Pt$ for some $t\in R$. In this case $R$ is self FP-injective and the converse holds. So, if $A$ is a proper Archimedean ideal then  $R/A$ is self FP-injective and it follows that $(A:(A:I))=I$ for each ideal $I\supseteq A$ which is not of the form $Pt$ for some $t\in R$.}
\end{remark}

\begin{lemma}
\label{L:topbot} Let $G$ be a FP-injective module over a chain ring $R$. Then $G^{\sharp}\subseteq Z\cap G_{\sharp}$ and $G$ is a module over $R_{G_{\sharp}}$.
\end{lemma}
\begin{proof}
Let $a\in R\setminus G_{\sharp}$ and $x\in G$. Let $b\in (0:a)$. Then $abx=0$, whence $bx=0$. So, $(0:a)\subseteq (0:x)$. It follows that $x=ay$ for some $y\in G$ since $G$ is FP-injective. Hence $a\notin G^{\sharp}$. If $a\notin Z$ then $0=(0:a)\subseteq (0:x)$ for each $x\in G$.
\end{proof}

\begin{proposition}
\label{P:main} Let $R$ be a chain ring, $E$ an FP-injective $R$-module and $U$ a pure submodule of $E$. Assume that $E_{\sharp}\subset Z$. Then $E/U$ is FP-injective.
\end{proposition}
\begin{proof} Let $E_{\sharp}=L$. Then $E$ and $U$ are modules over $R_L$.
Since $L\subset Z$, by \cite[Theorem 11]{Couch03} $R_L$ is coherent, whence $E/U$ is FP-injective.
\end{proof}

\begin{remark}
Let $R$ be a chain ring. Assume that $P$ is not finitely generated and not faithful. Then, for any indecomposable injective $R$-module $E$ and for any non-zero pure submodule $U$ of $E$, $E/U$ is FP-injective over $R/(0:P)$. 
\end{remark}
\begin{proof}
Since $P$ is not finitely generated and not faithful $R$ is not coherent. Let $R'=R/(0:P)$. Since $(0:P)$ is a non-zero principal ideal, $R'$ is coherent by \cite[Theorem 11]{Couch03}. First we assume that $E\ncong E(R/P)$. By \cite[Corollary 28]{Couch03} $E$ is an $R'$-module and it is easy to check that it is injective over $R'$ too. Hence $E/U$ is FP-injective over $R'$. Now suppose that $E=E(R)\cong E(R/P)$. Then $(0:P)$ is a submodule of $U$ and $E$. So, $E/U$ is the factor of $E/(0:P)$ modulo the pure submodule $U/(0:P)$. By \cite[Proposition 14]{Couch03} $E/(0:P)\cong E(R/Rr)$ for some $0\ne r\in P$. Hence $E/(0:P)$ is injective over $R'$. Again we conclude that $E/U$ is FP-injective over $R'$.
\end{proof}

The following example shows that $E/U$ is not necessarily FP-injective over $R$. 

\begin{example}
Let $D$ be a valuation domain whose order group is $\mathbb{R}$, $M$ its maximal ideal, $d$ a non-zero element of $M$ and $R=D/dM$. Assume that $D$ is not almost maximal. Then, for any indecomposable injective $R$-module $E$ and for any non-zero pure proper submodule $U$ of $E$, $E/U$ is not FP-injective over $R$. In particular, if $E=E(R)$, then $E/R$ is not FP-injective over $R$.
\end{example}
\begin{proof}
If $I$ is a non-zero proper ideal of $R$ then either $I$ is principal or $I=Pa$ for some $a\in R$. On the other hand $P$ is not finitely generated and not faithful. Let $x\in E\setminus U$. Then $(U:x)$ is not finitely generated. So, $(U:x)=Pb$ for some $b\in R$ and there exists $a\in P$ such that $Pb=(0:a)$. By lemma \ref{L:crucial} $E/U$ is not FP-injective over $R$.

Since $D$ is not almost maximal then $R$ is a proper pure submodule of its injective hull.
\end{proof}

\begin{lemma}
\label{L:arch} Let $R$ be a chain ring. Then:
\begin{enumerate}
\item $sI$ is Archimedean for each non-zero Archimedean ideal $I$ and for each $s\in P$ for which $sI\ne 0$;
\item $(A:I)^{\sharp}=I^{\sharp}$ for each Archimedean ideal $A$ and for each ideal $I$ such that $A\subseteq I$.
\end{enumerate}
\end{lemma}
\begin{proof}
$(1)$. Let $t\in R$ such that $tsI=sI$. If $b\in I$ then there exists $c\in I$ such that $sb=tsc$. If $sb\ne 0$, then by \cite[Lemma 5]{Couch03} $Rb=Rtc$. If $sb=0$, then $b\in (0:s)\subset tI$ since $tsI\ne 0$. So, $tI=I$. It follows that $t$ is invertible.

$(2)$. Let $J=I^{\sharp}$. First suppose $J\subset P$. Let $s\in R\setminus J$. Then $sI=I$. It follows that $(A:I)\subset Rs$. Let $r\in (A:I)$. Then $r=st$ for some $t\in R$. We have $tI=tsI=rI\subseteq A$. So, $t\in (A:I)$, $(A:I)=s(A:I)$ and $(A:I)^{\sharp}\subseteq J$. But since $A$ is  Archimedean we have $(A:(A:I))=I$ (Remark \ref{R:P=Z}). It follows that $(A:I)^{\sharp}=J$.

Now assume that $J=P$. If $P\subseteq (A:I)$ then $(A:I)^{\sharp}=P$. Now suppose that $(A:I)\subset P$. Let $s\in P\setminus (A:I)$. Therefore $((A:I):s)=(A:sI)\supset (A:I)$ since $A$ is Archimedean. Hence $(A:I)^{\sharp}=J=P$.
\end{proof}

\begin{lemma}
\label{L:ideal} Let $R$ be a chain ring, $I$ a non-zero Archimedean ideal of $R$ which is neither principal nor of the form $Pt$ for some $t\in R$, $0\ne a\in I$ and $A=I(0:a)$. Then:
\begin{enumerate}
\item If $(0:a)\subset (A:I)$ then there exists $c\in R$ such that $(A:I)=Rc$ and $(0:a)=Pc$;
\item $A$ is Archimedean if $Z=P$.
\end{enumerate}
\end{lemma}
\begin{proof}
$(1)$. Let $c\in (A:I)\setminus (0:a)$. It is easy to see that $A=cI$. Let $d\in (A:I)$ such that $c=td$ for some $t\in R$. Then $A=cI=tdI=dI$. From Lemma \ref{L:arch} we deduce that $t$ is invertible. So, $(A:I)=Rc$. By way of contradiction suppose there exists $d\in R$ such that $(0:a)\subset Rd\subset Rc$. As above we get that $(A:I)=Rd$. This contradicts that $(A:I)=Rc$. Hence $(0:a)=Pc$.

$(2)$. First we show that $A\subset c(0:a)$ if $c\in P\setminus I$. By way of contradiction suppose that $A=c(0:a)$. Since $I\ne Pt$ for each $t\in R$, by \cite[Lemma 29]{Couch03} there exists $d\in P$ such that $I\subset dcR$. We have $dc(0:a)=c(0:a)$. From $a\in I$ we deduce that $a=rdc$ for some $r\in P$. It follows that $rc(0:a)=rdc(0:a)=0$. But $rc\notin Ra$ implies that $rc(0:a)\ne 0$, whence a contradiction. Let $s\in P\setminus I$. Since $I$ is Archimedean there exists $t\in (I:s)\setminus I$. We have $A\subset t(0:a)\subseteq (A:s)$. Hence $A$ is Archimedean.
\end{proof}

\begin{lemma} \label{L:ann} Let $R$ be a chain ring such that $0\ne Z\subset P$ and
$A$ a non-zero Archimedean ideal.
\begin{enumerate}
\item if $A\subset rZ$ for some $r\in R$ then $(A:rZ)=Qs$ for some $s\in Z$;
\item  if $I$ is an ideal satisfying $I^{\sharp}=Z$, $A\subset I$ and $I\ne rZ$ for any $r\in R$, then $(A:I)\ne bQ$ for any $b\in Z$.
\end{enumerate}
\end{lemma}
\begin{proof} 
$(1)$. Let $J=(A:rZ)$. By Remark \ref{R:P=Z} $(A:J)=rZ$. By Lemma \ref{L:arch} $J^{\sharp}=Z$, so $J$ is an ideal of $Q$. By way of contradiction suppose that $J$ is not finitely generated over $Q$. Then $J=ZJ$ and $rJ=rZJ\subseteq A$. Whence $rR\subseteq (A:J)=rZ$. This is false. Hence $J=Qs$ for some $s\in R$.

$(2)$. By way of contradiction suppose that $(A:I)=bQ$ for some $b\in Z$. It follows that $bI\subset A$. So, $(bI:I)\subseteq (A:I)$. It is obvious that $b\in (bI:I)$ and since $I$ is a $Q$-module we have $(bI:I)=bQ$. Since $I\ne cZ$ for each $c\in Z$ we have $bI=\cap_{r\notin bI}rQ$ by \cite[Lemma 29]{Couch03}. Let $c\in A\setminus bI$. There exists $t\in Z$ such that $tc\notin bI$. We have $(Rtc:I)=(Rc:I)$. It is obvious that $(Rc:I)\subseteq (Rtc:tI)$. Let $r\in (Rtc:tI)$. For each $s\in I$ $ts=tcv$ for some $v\in R$. If $ts\ne 0$ then $Rs=Rcv$ by \cite[Lemma 5]{Couch03}. If $ts=0$ then $s\in (0:t)\subset Rc$ because $tc\ne 0$. Hence $(Rc:I)=(Rtc:tI)$. But $tI\ne I$ because $t\in Z=I^{\sharp}$. Since $Rtc$ is Archimedean we get that $(Rtc:I)\subset (Rtc:tI)$, whence a contradiction.
\end{proof}

Let $\widehat{R}$ be the pure-injective hull of $R$ and $x\in\widehat{R}\setminus R$. As in \cite{SaZa85} the breadth ideal  B($x$) of $x$ is defined as follows: B($x$)$ = \{ r\in R\mid x\notin R + r\widehat{R}\}$. 

\begin{proposition}  \label{P:breadth}  Let $R$ be a chain and $I$ a proper ideal of  $R$. Then:  
\begin{enumerate}
\item \cite[Proposition 20]{Couc06} $R/I$  is not complete in its f.c. topology if and only if there exists  $x\in\widehat{R}\setminus R$ such that   $I =\mathrm{B}(x)$;
\item \cite[Proposition 3]{Cou01} if $Z=P$ and $I =\mathrm{B}(x)$ for some $x\in\widehat{R}\setminus R$ then:
\begin{enumerate}
\item $I = (0 : (R : x))$;
\item $(R : x) = P(0 : I)$ and  $(R : x)$  is not finitely generated.
\end{enumerate}
\end{enumerate}
\end{proposition}

We say that a module $M$ is \textbf{polyserial} if it has a
pure-composition series
\[0=M_0\subset M_1\subset\dots\subset M_n=M,\]
i.e. $M_k$ is a pure submodule of $M$
and $M_k/M_{k-1}$ is a uniserial module for each $k=1,\dots, n$. 

The \textbf{Malcev rank} of a module $M$ is defined as the cardinal
number
\[\mathrm{Mr}\ M=\mathrm{sup}\{\mathrm{gen}\ X\mid X\ \mathrm{finitely\ generated\ submodule\ of}\ M\}.\] 

\begin{proposition}\label{P:unipoly}
Let $U$ be a submodule of a FP-injective module $E$ over a chain ring $R$. Then the following conditions are equivalent:
\begin{enumerate}
\item $E/U$ is FP-injective if $U$ is uniserial;
\item $E/U$ is FP-injective if $U$ is polyserial.
\end{enumerate}
\end{proposition}
\begin{proof}
It is obvious that only $(1)\Rightarrow (2)$ needs a proof. By \cite[Proposition 13]{Couc06} $\mathrm{Mr}\ U$ is finite and equals the length of any pure-composition series of $U$. Let $n=\mathrm{Mr}\ U$. Let $U_1$ be a pure uniserial submodule of $U$. Then $U/U_1$ is a pure submodule of $E/U_1$ which is FP-injective. On the other hand $U/U_1$ is polyserial and $\mathrm{Mr}\ U/U_1=n-1$. We conclude by induction on $n$. 
\end{proof}

\section{Chain ring case: main results}

\begin{lemma}
\label{L:uniserial} Let $R$ be a chain ring, $E$ an indecomposable injective $R$-module and $U$ a pure uniserial submodule of $E$. Then, for each $0\ne e\in E$ there exists a pure submodule $V$ of $E$ containing $e$.
\end{lemma}
\begin{proof}
There exists $r\in R$ such that $0\ne re\in U$. The purity of $U$ implies there exists $u\in U$ such that $re=ru$. By \cite[Lemma 2]{Couch03} $(0:e)=(0:u)$. Let $\alpha: Re\rightarrow U$ be the homomorphism defined by $\alpha(e)=u$. It is easy to check that $\alpha$ is a monomorphism. So, there exists a homomorphism $\beta: U\rightarrow E$ such that $\beta(\alpha(e))=e$. Then $\beta$ is injective since $\alpha$ is an essential monomorphism. Let $V=\beta(U)$. Thus by using the fact that a submodule of an injective module is pure if and only if it is a FP-injective module, we get that $V$ is a pure submodule of $E$.
\end{proof}

\begin{theorem}
\label{T:main} Let $R$ be a chain ring. Assume that $Q$ is not coherent. Consider the following two conditions:
\begin{enumerate}
\item $R/I$ is complete in its f.c. topology for each proper ideal $I$, $I\ne rZ$ for any $r\in R$, satisfying $I^{\sharp}=Z$;
\item for each FP-injective $R$-module $E$ and for each pure polyserial submodule $U$ of $E$, $E/U$ is FP-injective.
\end{enumerate}
Then $(1)\Rightarrow (2)$ and the converse holds if each indecompo\-sable injective module $E$ for which $E_{\sharp}=Z$ contains a pure uniserial submodule.
\end{theorem}
\begin{proof} $(1)\Rightarrow (2)$.
We may assume that $U$ is uniserial by Proposition \ref{P:unipoly} and that $E$ is injective and indecomposable by Proposition \ref{P:redu}. Let $E_{\sharp}=L$. By Proposition \ref{P:main} we may suppose that $Z\subseteq L$. After, possibly replacing $R$ with $R_L$, we may assume that $L=P$. Let $x\in E\setminus U$ and $a\in R$ such that $(0:a)\subseteq (U:x)$. Let $A=(0:x)$ and $R'=R/A$. Let $E'=\{y\in E\mid A\subseteq (0:y)\}$. Let $c\in (U:x)\setminus A$. Since $U$ is a pure submodule there exists $e\in U$ such that $cx=ce$ and by \cite[Lemma 2]{Couch03} $(0:e)=A$. By \cite[Lemma 26]{Couch03} $A^{\sharp}=E_{\sharp}=P$. Thus $A$ is Archimedean. Let $v\in U\cap E'$ such that $e=tv$ for some $t\in R$. Then $A\subseteq (0:v)=t(0:e)=tA$. So, $tA=A$. We deduce that $t$ is invertible and $R'\cong Re=E'\cap U$. It follows that $(U:x)=(Re:x)$. We have B($x$)$=I/A$ where either $I^{\sharp}\ne Z$ or $I=rZ$ for some $r\in R$. We deduce that $(Re:x)=P(A:I)$ by Proposition \ref{P:breadth}. By Lemma \ref{L:arch} $(A:I)^{\sharp}=I^{\sharp}$. If $(A:I)$ is not principal then $(Re:x)=(A:I)$. If $(A:I)=Rr$ for some $r\in P$, then $(Re:x)=Pr$, and in this case $(Re:x)^{\sharp}=P=I^{\sharp}$. In the two cases $(Re:x)^{\sharp}=I^{\sharp}$. We deduce that $(U:x)^{\sharp}\ne Z$ if $I^{\sharp}\ne Z$. If $I=rZ$ for some $r\in R$, then $R\ne Q$ and $Z\ne P$ because $Q/rZ$ is complete and $R/rZ$ is not. By Lemma \ref{L:ann} $(A:I)=Qs$ for some $s\in R$. But, since $R_{\sharp}=Z$, $(0:a)^{\sharp}=Z$  by \cite[Lemma 26]{Couch03}, and by \cite[Theorem 10]{Couch03} $(0:a)$ is not finitely generated over $Q$. Hence $(0:a)\subset (U:x)$. By Lemma \ref{L:crucial} there exist $u\in U$ and $y\in E$ such that $x=ay+u$, so, $E/U$ is FP-injective.

$(2)\Rightarrow (1)$.
By way of contradiction suppose there exists an ideal $I$ of $R$, $I\ne rZ$ for any $r\in Z$, such that $I^{\sharp}=Z$ and $R/I$ is not complete in its f.c. topology. Since the natural map $R\rightarrow Q$ is a monomorphism, as in \cite[Proposition 4]{Cou10}, we can prove that $Q/I$ is not complete in its f.c. topology. After, possibly replacing $R$ by $Q$ we may assume that $Z=P$. Then, $R$ is not coherent and $I$ is Archimedean. Let $s\in P\setminus I$. So, $I\subset (I:s)\subset P$. If $E$ is the injective hull of $R$, by Proposition \ref{P:breadth} there exists $x\in E\setminus R$ such that $I=$B($x$). Since $s\notin I$, $x=r+sy$ with $r\in R$ and $y\in E\setminus R$. We have B($y$)=$(I:s)$, whence $R/(I:s)$ is not complete too. So, possibly, after replacing $I$ with $(I:s)$, we can choose $I\ne 0$.

First assume that $I=Ra$ for some $a\in P$. Let $E$ be the injective hull of $R$ and $x\in E\setminus R$ such that B($x$)$=I$. By Proposition \ref{P:breadth} $(R:x)=P(0:a)=(0:a)$ since $(0:a)$ is not finitely generated by \cite[Theorem 10]{Couch03}. By Lemma \ref{L:crucial} $E/R$ is not FP-injective.

Now, suppose that $I$ is not finitely generated. Let $a$ be a non-zero element of $I$. Then $(0:I)\subset (0:a)$. So, if $A=I(0:a)$, then $A\ne 0$ and $A$ is Archimedean by Lemma \ref{L:ideal}. Let $R'=R/A$, $e=1+A$, $E$ the injective hull of $R'$ over $R$ and $E'=\{z\in E\mid A\subseteq (0:z)\}$. Then $E'$ is the injective hull of $R'$ over $R'$. By hypothesis and Lemma \ref{L:uniserial} $R'$ is  contained in a pure uniserial submodule $U$ of $E$. As in the proof of $(1)\Rightarrow (2)$ we get $R'=E'\cap U$. Let $I'=I/A$ and $P'=P/A$. Since $R'/I'$ is not complete in its f.c. topology there exists $x\in E'$ such that B($x$)$=I'$. Then $(R':_{R'}x)=P'(0:_{R'}I')$. It is easy to see that $(R':_{R'}x)=(U:x)/A$ and $(0:_{R'}I')=(A:I)/A$. So, $(U:x)=P(A:I)$. From Lemma \ref{L:ideal} we deduce that $P(A:I)=(0:a)$. Hence $(U:x)=(0:a)$, whence $E/U$ is not FP-injective by Lemma \ref{L:crucial}.
\end{proof}

\begin{theorem}
\label{T:mainCoh} Let $R$ be a chain ring such that $Z\ne 0$ . Assume that $Q$ is coherent. The following conditions are equivalent:
\begin{enumerate}
\item $R/Z$ is complete in its f.c. topology;
\item for each FP-injective $R$-module $E$ and for each pure polyserial submodule $U$ of $E$, $E/U$ is FP-injective.
\end{enumerate} 
\end{theorem}
\begin{proof}
$(1)\Rightarrow (2)$. We may assume that $Z\subset P$. Since $Q$ is coherent, for each $0\ne a\in Z$, $(0:a)=bQ$ for some $0\ne b\in Z$. Let $E$ be an injective module, $U$ a pure uniserial submodule of $E$ and $L=E_{\sharp}$. We may assume that $E$ is indecomposable. If $L\subseteq Z$ then $E/U$ is FP-injective because $R_L$ is coherent. Now assume that $Z\subset L$. As in the proof of Theorem \ref{T:main} we may suppose that $L=P$. Let $x\in E\setminus U$, $A=(0:x)$, $a\in R$ such that $(0:a)\subseteq (U:x)$ and $c\in (U:x)\setminus A$. As in the proof of Theorem \ref{T:main} we show there exists an ideal $I$ such that $(U:x)=P(A:I)$. If $I^{\sharp}\ne Z$ we do  in the proof of Theorem \ref{T:main} to show that $(0:a)\subset (U:x)$. Now suppose that $I^{\sharp}=Z$. By hypothesis $I\ne rZ$ for each $r\in R$. Since $(A:I)^{\sharp}=Z\ne P$, $(U:x)=(A:I)\ne (0:a)$ by Lemma \ref{L:ann}. We conclude by Lemma \ref{L:crucial} that $E/U$ is FP-injective.

$(2)\Rightarrow (1)$. Let $0\ne a\in Z$. Then $(0:a)=bQ$ for some $0\ne b\in Z$. It is obvious that $(bZ:Z)\subseteq (bR:Z)$. Let $c\in (bR:Z)$ then $cZ\subset bR$. Since $(bQ/bZ)$ is simple over $Q$ and $cZ$ is a proper $Q$-submodule of $bQ$ we get that $cZ\subseteq bZ$. Hence $(Rb:Z)=(bZ:Z)$. Since $bZ$ is an Archimedean ideal over $Q$ and that $(bZ:b)=Z$ then $(bZ:Z)=(bZ:(bZ:b))=bQ=(0:a)$ by Remark \ref{R:P=Z}. So, $(bR:Z)=(0:a)$. Now, assume that $R/Z$ is not complete in its f.c. topology. Let $E$ the injective hull of $R/bR$. By \cite[Corollary 22(3)]{Couch03} there exists a pure uniserial submodule $U$ of $E$ containing $e=1+bR$. Now, as in the proof of Theorem \ref{T:main} we show that there exists $x\in E\setminus U$ such that $(U:x)=(bR:Z)=(0:a)$. By Lemma \ref{L:crucial} $E/U$ is not FP-injective. This contradicts the hypothesis.
\end{proof}

\begin{corollary}
\label{C:Znonidem} Let $R$ be a chain ring such that $Z^2\ne Z$ . The following conditions are equivalent:
\begin{enumerate}
\item $R/Z$ is complete in its f.c. topology;
\item for each FP-injective $R$-module $E$ and for each pure polyserial submodule $U$ of $E$, $E/U$ is FP-injective.
\end{enumerate} 
\end{corollary}
\begin{proof} Since $Z^2\ne Z$, $Z$ is principal over $Q$ and $Q$ is coherent by \cite[Theorem 10]{Couch03}.
\end{proof}

A chain ring $R$ is said to be {\bf strongly discrete} if $L\ne L^2$ for each non-zero prime ideal of $R$.

\begin{corollary}
\label{C:discr} Let $R$ be a strongly discrete chain ring. The following conditions are equivalent:
\begin{enumerate}
\item $R/Z$ is complete in its f.c. topology;
\item for each FP-injective $R$-module $E$ and for each polyserial submodule $U$ of $E$, $E/U$ is FP-injective.
\end{enumerate}
\end{corollary}

\begin{corollary}
\label{C:Z=P} Let $R$ be a chain ring such that $Z=P$. Consider the following conditions:
\begin{enumerate}
\item either $R$ is coherent or $R/I$ is complete in its f.c. topology for each Archime\-dean ideal $I$;
\item for each FP-injective $R$-module $E$ and for each pure polyserial submodule $U$ of $E$, $E/U$ is FP-injective.
\end{enumerate}
Then $(1)\Rightarrow (2)$ and the converse holds if each indecomposable injective module $E$ for which $E_{\sharp}=P$ contains a pure uniserial submodule.
\end{corollary}
\begin{proof}
It is a consequence of Theorem \ref{T:main}.
\end{proof}

For each module $M$ we denote by $\mathcal{A}(M)$ its set of annihilator ideals, i.e. an ideal $A$ belongs to $\mathcal{A}(M)$ if there exists $0\ne x\in M$ such that $A=(0:x)$. If $E$ is a uniform injective module over a chain ring $R$, then, for any $A,\ B\in\mathcal{A}(E),\ A\subset B$ there exists $r\in R$ such that $A=rB$ and $B=(A:r)$ (see \cite{Nis72}).

\begin{lemma} \label{L:puruni} Let $R$ be a chain ring. Assume that $Z_1\ne Z$, where $Z_1$ is the union of all prime ideals properly contained in $Z$. Let $E$ be an indecomposable injective $R$-module and $0\ne e\in E$. Suppose that $E_{\sharp}=Z$. Then $E$ contains a uniserial pure submodule $U$ such that $e\in U$. 
\end{lemma}
\begin{proof} Since $E$ is a module over $Q$, we may assume that $R=Q$. By Lemma \ref{L:uniserial} it is enough to show that $E$ contains a pure uniserial submodule. Since $R/Z_1$ is archimedean, $P$ is countably generated by \cite[Lemma 33]{Couch03}. By \cite[Proposition 32]{Couch03} $(0:P)$ is a countable intersection of ideals containing it properly. So, by  \cite[Proposition 19]{Couch03} $\mathrm{E}(R/P)$ and $\mathrm{E}(R/rR)$, $r\ne 0$, contain a pure uniserial submodule. If $\mathcal{A}(E)=\mathcal{A}(R)$ then $E\cong\mathrm{E}(R)$. Since $R$ is self FP-injective, it follows that $E$ contains a pure uniserial submodule. Now assume that $\mathcal{A}(E)\ne\mathcal{A}(R)$ and $\mathcal{A}(E)\ne \{rR\mid r\in R\}$. By \cite[Theorem 5.5]{ShLe74} there exists a uniserial $R$-module $U$ such that $\mathcal{A}(U)=\mathcal{A}(E)$ and consequently $E\cong\mathrm{E}(U)$. Let $r\in R$ and $u\in U$ such that $(0:r)\subseteq (0:u)$. Then $(0:r)\subset (0:u)$, and $r(0:u)$  is not a principal ideal. So, $(0:P)\subset r(0:u)$, and by \cite[Proposition 27]{Couch03} there exists $v\in U$ such that $(0:v)\subset r(0:u)\subset (0:u)$. It follows that $u=tv$ for some $t\in R$. By \cite[Lemma 2]{Couch03} $(0:v)=t(0:u)\subset r(0:u)$. Hence $t\in rR$ and $u\in rU$. We conclude that $U$ is FP-injective, whence it is isomorphic to a pure submodule of $E$.
\end{proof}

In the following theorems let us observe that the word "polyserial" is replaced with "of finite Goldie dimension".

\begin{theorem}
\label{T:Z=P1} Let $R$ be a chain ring such that $Z=P$. Assume that $P\ne P_1$ where $P_1$ is the union of all nonmaximal prime ideals of $R$. The following conditions are equivalent:
\begin{enumerate}
\item either $R$ is coherent or $R/P_1$ is almost maximal;
\item for each FP-injective $R$-module $E$ and for any its pure submodule $U$ of finite Goldie dimension, $E/U$ is FP-injective.
\end{enumerate}
\end{theorem}
\begin{proof} $(2)\Rightarrow (1)$. By Lemma \ref{L:puruni} each indecomposable injective module $E$ for which $E^{\sharp}=P$ contains a pure submodule. We conclude by Corollary \ref{C:Z=P}.

$(1)\Rightarrow (2)$. We may assume that $R$ is not coherent and $E$ is the injective hull of $U$. Then $E$ is a finite direct sum of indecomposable injective modules. So, it is easy to show that $E=F\oplus G$ where $F_{\sharp}=P$ and $G_{\sharp}=L\subset P$. If $F=0$ then $E$ and $U$ are modules over $R_L$ which is coherent by \cite[Theorem 11]{Couch03}. In this case $E/U$ is FP-injective. Now, $F\ne 0$. Let $a\in R$ and $x\in E\setminus U$ such that $(0:a)\subseteq (U:x)$. We have $x=y+z$ where $y\in F$ and $z\in G$. By \cite[Lemma 26]{Couch03} $(0:y)^{\sharp}=P$ and $(0:z)^{\sharp}\subseteq L$. It is obvious that $(0:x)=(0:y)\cap (0:z)$. So, it is possible that $(0:x)^{\sharp}\subseteq L$. Let $B$ be the kernel of the natural map $R\rightarrow R_L$. For any $s\in P\setminus L$ we have $(0:P)\subset (0:s)\subseteq B\subseteq (0:x)$. By \cite[Corollary 28]{Couch03} there exists $f\in F$ such that $(0:f)\subset (0:x)$. There exists $b\in R$ such that $0\ne bf\in U$. Since $U$ is a pure submodule there exists $u\in U$ such that $bf=bu$. By \cite[Lemma 2]{Couch03} $(0:u)=(0:f)$. It is obvious that $(U:x+u)=(U:x)$ and we have $(0:x+u)=(0:u)$ and $(0:x+u)^{\sharp}=P$. So, after possibly replace $x$ with $x+u$, we may assume that $(0:x)^{\sharp}=P$. For any $c\in (U:x)\setminus (0:x)$ let $e_c\in U$ such that $ce_c=cx$. Then $E$ contains an injective hull $E_c$ of $Re_c$, and clearly $x\in E_c$. So, we do as in the proof of Theorem \ref{T:main} to show that $(Re_c:x)^{\sharp}\ne P$ for any $c\in (U:x)\setminus (0:x)$. It is obvious that $(U:x)=\cup_{c\in (U:x)\setminus (0:x)}(Re_c:x)$. It follows that $(U:x)^{\sharp}\subseteq\cup_{c\in (U:x)\setminus (0:x)}(Re_c:x)^{\sharp}\subseteq P_1$. We conclude as in the proof of Theorem \ref{T:main}.
\end{proof}

\begin{theorem}
\label{T:main2} Let $R$ be a chain ring. Assume that $Z\ne Z_1$ where $Z_1$ is the union of all prime ideals properly contained in $Z$. Suppose $R/Z_1$ is almost maximal. Then, for each FP-injective $R$-module $E$ and for any its pure submodule $U$ of finite Goldie dimension, $E/U$ is FP-injective.
\end{theorem}
\begin{proof}
We may assume that $R$ is not coherent and $E$ is the injective hull of $U$. By Theorem \ref{T:Z=P1} we suppose that $Z\ne P$. As in the proof of Theorem \ref{T:main} we may assume that $E_{\sharp}=P$. Let $a\in R$ and $x\in E\setminus U$ such that 
$(0:a)\subseteq (U:x)$. It is possible that $(0:x)=0$. But, there exists $b\in R$ such that $0\ne bx\in U$, and since $U$ is a pure submodule there exists $v\in U$ such that $bx=bv$. We get $(0:x-v)\ne 0$ and $(U:x-v)=(U:x)$. Now we do the same proof as in Theorem \ref{T:Z=P1} to conclude.
\end{proof}

\begin{corollary}
\label{C:Arch} Let $R$ be an Archimedean chain ring. The following conditions are equivalent:
\begin{enumerate}
\item $R$ is either coherent or maximal;
\item for each FP-injective $R$-module $E$  and for each pure submodule $U$ of finite Goldie dimension of $E$, $E/U$ is FP-injective.
\item for each injective $R$-module $E$  and for each pure submodule $U$ of finite Goldie dimension of $E$, $E/U$ is injective.
\end{enumerate}
\end{corollary}
\begin{proof}
$(1)\Rightarrow (2)\ \mathrm{and}\ (3)$. By Proposition \ref{P:redu} we may assume that $E$ is injective of finite Goldie dimension. If $R$ is maximal then $E$ is a finite direct sum of uniserial modules by \cite[Theorem]{Gil71}. By \cite[Theorem XII.2.3]{FuSa01} (this theorem holds even if $R$ is not a domain) $U$ is a direct summand of $E$. So, $U$ and $E/U$ are injective. If $R$ is coherent we apply \cite[Lemma 3]{Cou05}.

$(2)\Rightarrow (1)$ by Theorem \ref{T:Z=P1}. 
\end{proof}

\begin{corollary}
\label{C:locArch} Let $R$ be an arithmetical ring such that $R_P$ is Archimedean for any maximal ideal $P$ of $R$. Then the following conditions are equivalent:
\begin{enumerate}
\item $R_P$ is either coherent or maximal for each maximal ideal $P$ of $R$;
\item for each FP-injective $R$-module $E$  and for each pure submodule $U$ of finite Goldie dimension of $E$, $E/U$ is FP-injective;
\item for each injective $R$-module $E$  and for each pure submodule $U$ of finite Goldie dimension of $E$, $E/U$ is injective.
\end{enumerate}
\end{corollary}
\begin{proof}
By Corollary \ref{C:Arch} $(3)\Rightarrow (1)$. 

$(1)\Rightarrow (2)$. We may assume that $E$ is injective of finite Goldie dimension. By \cite[Corollary 4]{Cou09} $E_P$ is injective, and $U_P$ is a pure submodule of $E_P$. We must prove that $(E/U)_P$ is FP-injective for each maximal ideal $P$ of $R$. If $R_P$ is coherent it is a consequence of Proposition \ref{P:locoh}. If $R_P$ is maximal and non coherent, first we show that $E_P$ is a finite direct sum of indecomposable injective $R_P$-modules. We may assume that $E$ is indecomposable. Since $\mathrm{End}_R(E)$ is local, there exists a maximal ideal $L$ such that $E$ is a module over $R_L$. If $L=P$ then $E_P=E$. If $L\ne P$ then $E_P=0$ because $P$ is also a minimal prime ideal. By Corollary \ref{C:Arch} $(E/U)_P$ is FP-injective. We conclude that $E/U$ is FP-injective.

$(2)\Rightarrow (3)$. We have $E=E_1\oplus\dots\oplus E_n$ where $E_k$ is indecomposable for $k=1,\dots,n$. For $k=1,\dots,n$, let $P_k$ be the maximal ideal of $R$ which verifies that $E_k$ is a module over $R_{P_k}$. If $S=R\setminus (P_1\cup\dots\cup P_n)$, then $E$ and $U$ are modules over $S^{-1}R$. So, we replace $R$ with $S^{-1}R$ and we assume that $R$ is semilocal. By \cite[Theorem 5]{Jen66} each ideal of $R$ is principal ($R$ is B\'ezout). By using \cite[Corollary 36]{Couch03} it is easy to prove that each ideal of $R$ is countably generated. So, we can do the same proof as in \cite[Lemma 3]{Cou05} to show that $E/U$ is injective.
\end{proof}

\section*{Acknowledgements}

This work was presented at the "International Workshop on Algebra and Applications" held in Fes-Morocco, june 18-21, 2014. I thank again the organizers of this workshop. I thank too the laboratory of mathematics Nicolas Oresme of the university of Caen Basse-Normandie which allowed me to participate to this workshop.



\end{document}